\documentclass[10pt]{amsart}
\usepackage{graphicx, amsmath, fullpage,  amssymb, amsthm, amsfonts, color, mathrsfs, eucal, moreverb, mathtools, float}
\usepackage{enumerate}
\newtheorem{theorem}{Theorem}[section]
\newtheorem{lemma}[theorem]{Lemma}
\newtheorem{proposition}[theorem]{Proposition}
\newtheorem{corollary}[theorem]{Corollary}

\theoremstyle{definition}

\usepackage[svgnames]{xcolor}
\theoremstyle{remark}
\newtheorem{remark}[theorem]{Remark}
\numberwithin{equation}{section}
\usepackage[colorlinks,
            linkcolor=FireBrick,
            anchorcolor=red,
            citecolor=RoyalBlue
            ]{hyperref}
 
\newcommand{\ER}{Erd\H{o}s-R\'enyi~}

\newcommand{\ic}{\mathrm{i}}

\newcommand{\RE}{\mathfrak{Re}}
\DeclareMathSymbol{I}{\mathalpha}{operators}{`I}
\newcommand{\spec}{\mathrm{Spec}}

\newcommand{\diag}{\mathrm{diag}}
\begin{document}

\title{Eigenvalues of the non-backtracking operator detached from the bulk}
\author{Simon Coste}
\address{INRIA Paris, DYOGENE team \\ Office C330} 
\email{simon.coste@inria.fr}

\author{Yizhe Zhu}
\address{Department of Mathematics, University of California, San Diego, La Jolla, CA 92093}
\email{yiz084@ucsd.edu}

%
\date{\today}
\thanks{Y.Z. is partially supported by NSF DMS-1712630.}

\begin{abstract}  
We describe the non-backtracking spectrum of a stochastic block model with connection probabilities $p_{\mathrm{in}}, p_{\mathrm{out}} = \omega(\log n)/n$. In this  regime  we answer a question posed in \cite{dall2019optimized} regarding the existence of a real  eigenvalue `inside' the bulk, close to the location  $\frac{p_{\mathrm{in}}+ p_{\mathrm{out}}}{p_{\mathrm{in}}- p_{\mathrm{out}}}$. We also introduce a variant of the Bauer-Fike theorem well suited for perturbations of quadratic eigenvalue problems, and which could be of independent interest.
\end{abstract}

\maketitle

\section{Introduction}

For any real matrix $A$ with size $n \times n$, its non-backtracking operator $B$ is the real matrix  {indexed by the coordinates of the non-zero entries} of $A$, and is defined by 
\begin{equation}\label{def:B}
B_{(i,j), (k,\ell)} = A_{k,\ell} \mathbf{1}_{j=k, i \neq \ell}. 
\end{equation}
The non-backtracking matrix of a graph is the non-backtracking matrix of its adjacency matrix, and it is closely related to the Zeta function of the graph \cite{terras2010zeta}. Its spectrum was first studied in the case of  finite graphs and their universal covers \cite{hashimoto1989zeta,bass1992ihara,stark1996zeta,kotani2000zeta,terras2010zeta,angel2015non}. Recently, the non-backtracking operator attracted a lot of attention from random graph theory as a very powerful tool. In the spectral theory of random graphs, it was a key element in a new proof of the Alon-Friedman theorem for random regular graphs \cite{bordenave2015new}. In the same vein it has been used later to study the eigenvalues of random regular hypergraphs \cite{dumitriu2019spectra}, random bipartite biregular graphs \cite{brito2018spectral} and homogeneous or inhomogeneous Erd\H{o}s-R\'{e}nyi graphs \cite{krzakala2013spectral,wang2017limiting,bordenave2018nonbacktracking,gulikers2017non, benaych2017spectral, alt2019extremal, benaych2019largest}. Very recently, the real eigenvalues were used to prove estimates on the vector-colouring number of a graph \cite{2019arXiv190702539B}. 

Most of the results focus on the   {eigenvalues of large magnitude}, those which lie outside the bulk of the spectrum. They are known to be the `most informative' eigenvalues, as they capture some essential features about the structure of the graph. For instance, in community detection, the appearance of certain outliers indicates when the community structure can be recovered \cite{bordenave2018nonbacktracking, krzakala2013spectral}. A cornerstone result was that even in the difficult \textit{dilute} case, where the connection probabilities are of order $1/n$, reconstruction was feasible (under some condition) by looking at the eigenvalues of the non-backtracking matrix appearing outside the `bulk' of eigenvalues.

\begin{figure}\centering
\includegraphics[width=0.45\textwidth]{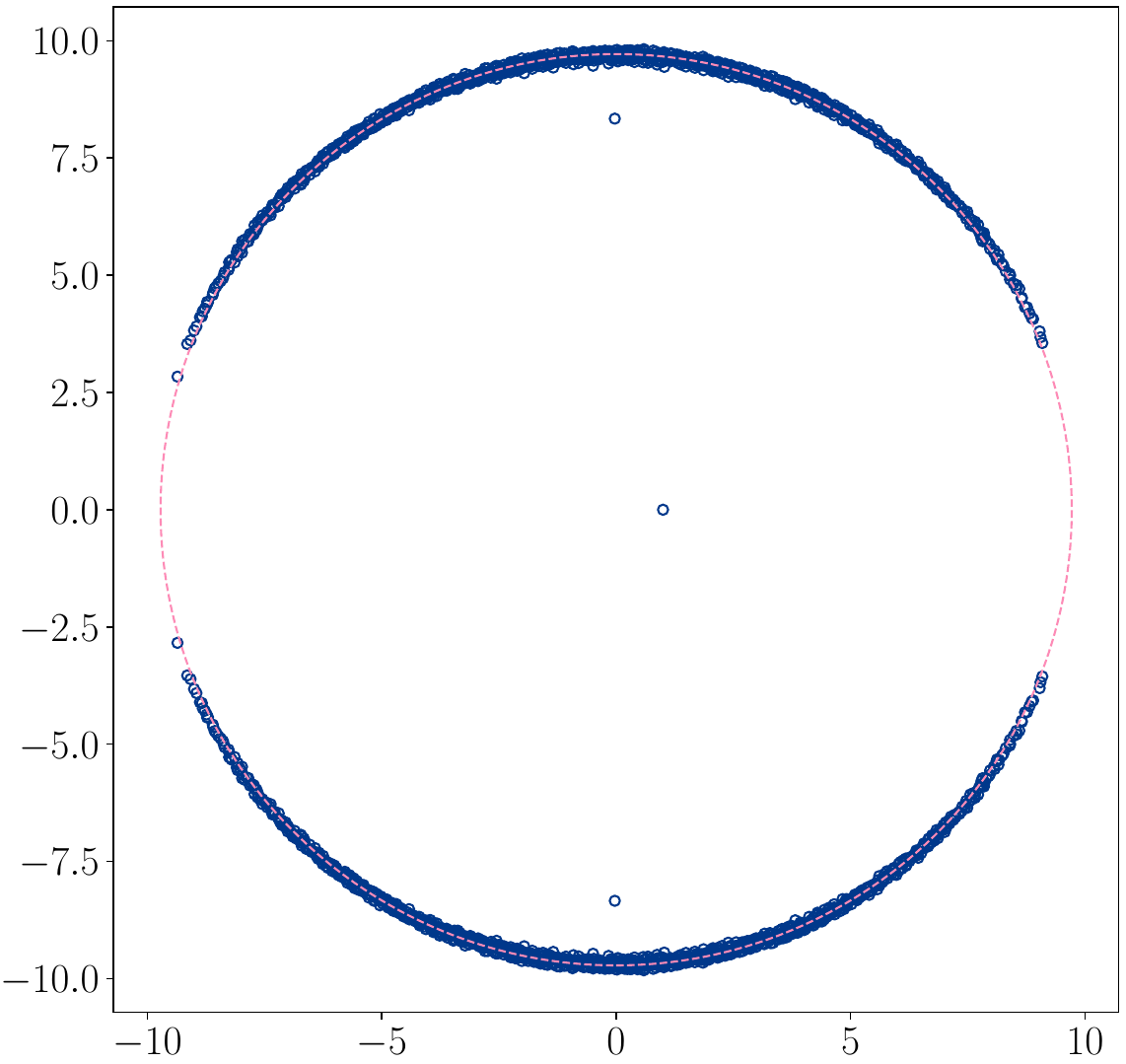}
\includegraphics[width=0.45\textwidth]{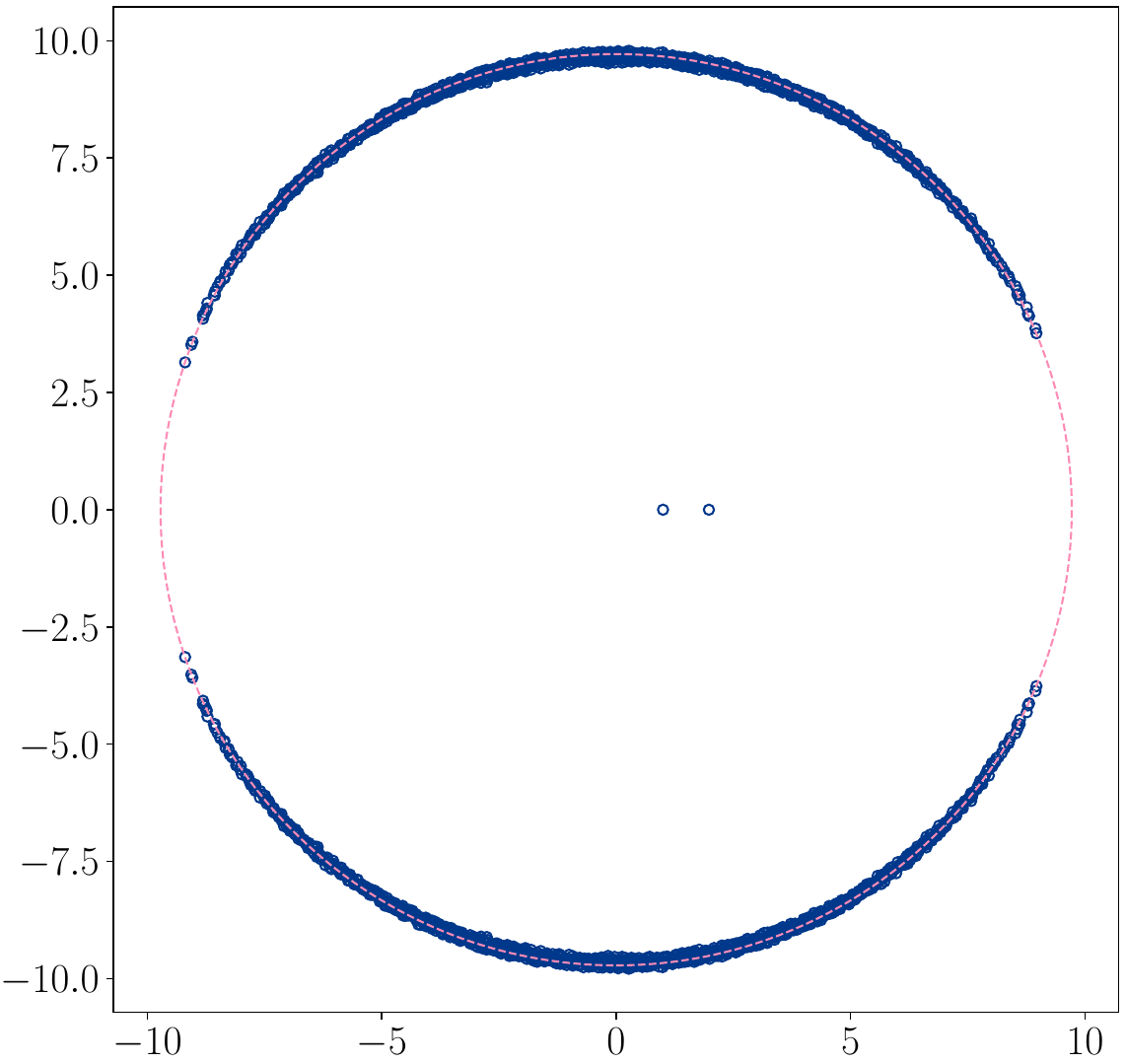}
\caption{The non-backtracking spectrum of two SBM graphs. The circle of radius $\sqrt{\alpha-1}$, where $\alpha$ is the mean degree, is drawn in pink. We have cropped the pictures, the outliers do not appear. On the left panel we took $p=q=(\log n)^2/n$, a classical \ER graph. On the right panel, we took $p=3(\log n)^2/n$ and $q=(\log n)^2/n$, which corresponds to a two-block SBM. We clearly see the two bulk insiders around $1$ and $2$.  In both cases the number of vertices is $n=1000$. }\label{figure}
\end{figure}

It was recently observed in \cite[Section 3.2]{dall2019optimized} that in fact,  there is a real eigenvalue isolated \textit{inside} the bulk that corresponds to the ratio of the two largest eigenvalues \emph{outside} the bulk,  {as displayed in the right panel} of Figure \ref{figure}.  Recall that $B$ is non-Hermitian with a complex spectrum, so that `inside'  the bulk is understood as eigenvalues inside the circle of the spectrum.  To the best of our knowledge, this phenomenon has not been rigorously studied yet. In this paper,  we prove the existence of this real eigenvalue inside the bulk for the stochastic block model (SBM)  {in the regime where the mean degree goes to infinity faster than $\log n$}.

\subsection*{Notations} Throughout the paper, we will adopt the conventional notations $a_n=o(b_n)$ when $\lim_{n\to\infty} \frac{a_n}{b_n} = 0$, $a_n=\omega(b_n)$ when $\lim_{n\to\infty} \frac{a_n}{b_n} = \infty$ and $a_n=O(b_n)$ when $|a_n/b_n|$ is bounded. All the results depend on the parameter $n$, the size of the graph, which is seen as large through $n \to \infty$.  {For any matrix $M$, we denote the spectral norm of $M$ as $\|M\|$.}

\subsection{Setting: the SBM in the logarithmic regime}

Consider a stochastic block model $G(n,p,q)$ with an even number $n$ of vertices, two blocks of equal size $n/2$, and   {two probability parameters $p,q$}: if $i,j$ are vertices in the same block then they are connected with probability $p$, and if they are in different blocks they are connected with probability $q$. We will place ourselves under the regime 
\begin{align}\label{condition:sparsity}\tag{A}
	\ p,q=\frac{\omega(\log n)}{n}, \qquad p,q\to 0,\qquad C_1\leqslant \frac{|p-q|}{p+q} \leqslant C_2
\end{align}
for some constants $C_1,C_2\in (0,1)$.  {The last condition is technical: we will see that it is  here to ensure that two separated outliers appear in the spectrum of the adjacency matrix, and are of the same order. This assumption is crucial in our perturbation analysis, see Remark \ref{rmk:sparsity} for further discussion.}

 {It is known (see for example \cite[Chapter 3]{bollobas2001random}) that when $q=p=\omega(\log n/n)$, the $G(n,p,p)$ graph (the Erd\H{o}s-R\'enyi model) is `almost regular' in the sense that all degrees are concentrated around $(n-1)p$.} In general, it can  {be shown} that when $p,q=\omega(\log n)/n$ then $G(n,p,q)$ is almost regular with degrees  {concentrated around} $$\frac{np}{2}-p+\frac{nq}{2}=\frac{n(p+q)}{2}-p. $$
See Subsection \ref{sec:degrees} for a proof. For this reason, we will denote by $\alpha$ the mean degree and by $\beta$ the mean difference degree:  
\begin{align*}
\alpha& := \mathbb E [d_i]=\frac{n(p+q)}{2}-p,\\
\beta& := \mathbb E[d^{\mathrm{in}}_i]-\mathbb E [d^{\mathrm{out}}_i]=(\frac{n}{2}-1)p-\frac{n}{2}q=\frac{n}{2}(p-q)-p
\end{align*}
where $d_i$ is the number of neighbors of vertex $i$, and $d^{\mathrm{out}}_i$ (resp. $d^{\mathrm{in}}_i$) is the number of neighbors of vertex $i$ which do not have the same type as $i$ (resp. which have the same type as $i$).  {Under assumption \eqref{condition:sparsity}, $\beta$ can be either positive or negative, and $\alpha,\beta$ are of the same order.}

Our assumptions \eqref{condition:sparsity} imply that the mean degree $\alpha$ is  {$\omega(\log n)$} and the mean difference degree $\beta$ has the same order as $\alpha$. Finally, the adjacency matrix of the graph is the $n \times n$ matrix $A$ defined by $A_{i,j} = \mathbf{1}_{i \sim j}$, where $i\sim j$ denotes the event that $i$ and $j$ are connected. 

\subsection{Main results}
Let $G=(V,E)$ be any finite graph with adjacency matrix $A$. The  Ihara-Bass   formula  gives a connection between the spectrum of $B$ defined in \eqref{def:B} and a \emph{quadratic eigenvalue problem}: for any complex $z$, 
\begin{equation}\label{NB_spec}
 \det(B-z I)=(z^2-1)^{|E|-|V|} \det (z^2I-z A+D-I), 
\end{equation}
see \cite{bass1992ihara,kotani2000zeta,stark1996zeta}. The zeros of the polynomial $z \mapsto \det (z^2I-z A+D-I)$ are usually called, with an abuse of language, \textit{the non-backtracking spectrum} of $B$, the two additional eigenvalues $\pm 1$ appearing with multiplicity $|E|-|V|$ in \eqref{NB_spec} being usually called \textit{trivial}. The non-backtracking spectrum can be expressed as the eigenvalues of a smaller matrix $H$: 
\begin{align}	\label{def:H}
H=\begin{bmatrix}
	A & I-D\\
	I & 0
\end{bmatrix}
\end{align}
where $D = \mathrm{diag}(A\mathbf{1})$ is the diagonal  degree matrix. This representation of the spectrum of $B$ in terms of $H$ is extremely useful: to compute the spectrum of $B$, we do not have to construct the matrix $B$, and we can analyze the spectrum of $B$ directly from $H$. 

Using these facts, we answer the question posed in \cite{dall2019optimized} regarding the existence of an isolated eigenvalue inside the bulk, at least under the assumptions \eqref{condition:sparsity}. We also give a detailed description of the non-backtracking spectrum that is similar to the one given in \cite{wang2017limiting} for \ER graphs.

\begin{theorem}\label{thm:main} Let $B$ be the non-backtracking operator of a stochastic block model $G(n,p,q)$ satisfying assumption \eqref{condition:sparsity}. We order the $2|E|$ eigenvalues of $B$ by decreasing modulus: $|\lambda_1(B) |\geqslant \dotsb \geqslant |\lambda_{2|E|}(B)|$. 

With probability $1-o(1)$, the spectrum of $B$ can be described as follows. First, the smallest eigenvalues  {in modulus} are the trivial eigenvalues $-1$ and $1$, each with multiplicity $|E|-|V|$. 

Then, in the non-trivial eigenvalues $\lambda_1(B), \dotsc, \lambda_{2n}(B)$, there are four real eigenvalues which are isolated, two `outliers'
\begin{equation}\label{eq:outsiders}
\lambda_1(B) = \alpha+O(\alpha^{3/4}),  \quad \lambda_2(B)=\beta+O(\alpha^{3/4}),
\end{equation}
and two `insiders'
\begin{equation}\label{eq:insiders}
\lambda_{2n-1}(B) = \frac{\alpha}{\beta}+o(1),   \quad\text{ and } \quad \lambda_{2n}(B)=1.
\end{equation}
All the other eigenvalues $\lambda_k(B)$ with $k \in \{3, \dotsc, 2n-3\}$ are located within distance $o(\sqrt{\alpha})$ of a circle of radius $\sqrt{\alpha-1}$. Moreover,  the real parts of eigenvalues of $\frac{B}{\sqrt{\alpha}}$ are asymptotically distributed as the semi-circle distribution supported on $[-1,1]$.
\end{theorem}

 {To present our approach in the most efficient and clear way}, we state and prove the theorem in the simplest regime, when there are only two blocks and when the community structure appears in the spectrum through the presence of an extra outlier near $\beta$. It is straightforward to check that our proof can be extended to a diversity of settings, when the mean degree is $\omega(\log n)$. We state this as an informal result: 

\emph{ Assume that with high probability the degree of each vertex is $(1+o(1))\alpha$  and the spectrum of the adjacency matrix has $k=O(1)$ outliers far outside $[-2\sqrt{\alpha}, 2\sqrt{\alpha}]$, say $\lambda_1 \approx \alpha$, and $\lambda_2, \dotsc, \lambda_k$.  {Assume $\lambda_1,\dots,\lambda_k$ are  are of the same order}. Then with high probability its non-backtracking spectrum will have $k$ eigenvalues near $\lambda_i$ for $i \in [k]$, then $k$ eigenvalues near $\alpha/\lambda_i$, and all the other eigenvalues will be located within distance $o(\sqrt{\alpha})$ of the circle of radius $\sqrt{\alpha-1}$}. 

\begin{remark}
     {The concentration results we have in Section \ref{sec:SBM} work for general inhomogeneous random graphs with outliers of the same order. Our modified Bauer-Fike theorem given in Theorem \ref{thm:BFQEP} works for general  inhomogeneous random graphs as well, as long as each vertex has almost regular degree $(1+o(1))\alpha$. Therefore all the analysis in Section \ref{sec:proof} can be extended to the case we mentioned above for $k$ outliers.} 
\end{remark}

\bigskip

The presence of bulk insiders in Theorem \ref{thm:main}  and in the preceding statement are illustrated in Figure \ref{figure} for a realization of an \ER graph and a realization of an SBM graph. 
Note that the description of the spectrum of $B$ in the preceding theorem is much more precise than Theorem 1.5 in \cite{wang2017limiting}.  {This comes from the fact that their perturbation parameter $R = c \sqrt{\log n/ p}$ goes to infinity (see Theorem 1.5 in \cite{wang2017limiting} for the exact statement, where the scaling parameter is different from ours)}. Our method includes a tailored version of the Bauer-Fike theorem suited for perturbations of matrices like \eqref{def:H}, which yields perturbation bounds that are  better than  the classical Bauer-Fike theorem in terms of the order of magnitude, and without which the existence of the two eigenvalues at $1$ and near $\beta/\alpha$ would not follow. We think  such variants of the Bauer-Fike theorem could be of independent interest.

\subsection{Bulk insiders and community detection}

The real eigenvalue of $B$ inside the bulk is closely related to  community detection problems for SBMs. An interesting
heuristic spectral algorithm  based on the Bethe Hessian matrix was proposed in \cite{watanabe2009graph,saade2014spectral}. The Bethe Hessian matrix, sometimes called  deformed Laplacian (\cite{grindrod2018deformed,dall2019optimized}), is defined as 
\[H(r) := (r^2-1)I+D-rA,\]	
 where $r\in \mathbb R$ is a regularizer to be carefully tuned. It is conjectured in \cite{saade2014spectral} that a spectral algorithm based on the eigenvectors associated with the negative eigenvalues of $H(r)$ with $r=\sqrt{\rho(B)}$ is able to reach the information-theoretic threshold confirmed in \cite{massoulie2014community,bordenave2018nonbacktracking,mossel2018proof,mossel2015reconstruction} for community detection in the dilute regime. In a subsequent work \cite{dall2019optimized}, the authors crafted a spectral algorithm based on $H(r)$ with $r=\alpha/\beta$  and empirically showed it outperforms already known spectral algorithms. Their choice of $r$ was motivated by the conjectured value of the real eigenvalue inside the bulk. The gain in using $H(r)$ instead of $B$ in spectral algorithms mainly comes from the fact that $H(r)$ has a smaller dimension than $B$, is Hermitian, and is easiest to build from $A$ --- nearly no preprocessing is needed, in contrast with non-backtracking matrices \cite{bordenave2018nonbacktracking,krzakala2013spectral}, self-avoiding path matrices \cite{massoulie2014community} or graph powering matrices \cite{abbe2018graph}. 

The relation between the Bethe Hessian matrix and the non-backtracking operator is given by the Ihara-Bass formula (see \eqref{NB_spec} above). Therefore, a good understanding of the real eigenvalues  of the non-backtracking operator is the first step towards  understanding the theoretical guarantee of the heuristic algorithms purposed in \cite{saade2014spectral,dall2019optimized}.

Unfortunately, our proof techniques do not work in the dilute regime. In the regime studied in this paper, community detection problems are now very well understood and clustering based on the second eigenvector of $A$ has been proven to yield \emph{exact reconstruction} (see \cite{abbe2017entrywise}). Our result should instead be seen as a preliminary step in view of 
1) proving the existence of bulk insiders in the dilute regime, 2) showing their usefulness in practical reconstruction.  {It will be helpful in practice to have a better understanding of the eigenvectors for the Bethe Hessian matrix, and we leave it as a future direction.}

The key obstacle is the lack of concentration of degrees profiles, which tells us  random graphs with bounded expected degrees are far away from being `roughly' regular (see also the discussion in Subsection \ref{sec:degrees}). Without this property, our perturbation analysis does not apply.

\subsection*{Organization of the paper}In Section \ref{sec:prelim}, we first state some classical facts on the non-backtracking spectrum of graphs then we state and prove a perturbation theorem which is well suited for quadratic eigenvalue problems and  improves the classical Bauer-Fike results. In Section \ref{sec:SBM}, we gather several facts on stochastic block models. Then we study the spectrum of $H$ as in \eqref{def:H} and a suitably chosen perturbation of $H$ (defined later in \eqref{eq:defH_0}). In Section \ref{sec:proof} we prove the main theorem.

\section{Perturbation of the non-backtracking spectrum}\label{sec:prelim}

\subsection{The non-backtracking spectrum}
When the graph is regular with degree $d$,  the diagonal matrix satisfies  $D=dI$, and we can relate the eigenvalues of $B$ with the eigenvalues of $A$ through exact algebraic relations as in the following elementary lemma.

\begin{lemma}\label{lem:H0}Let $\hat A$ be a Hermitian matrix with eigenvalues $\lambda_1, \dotsc, \lambda_n$ and let $\eta$ be a nonzero complex  number. Then, the characteristic polynomial of the matrix 
\[
\hat H_0=\begin{bmatrix}
	\hat A & \eta  I\\
	I & 0
\end{bmatrix}
\]
is given by $\chi_{\hat H_0}(z)=\prod_{k=1}^n (z^2-\lambda_k z - \eta)$, and the eigenvalues of $\hat H_0$ are the $2n$ complex numbers (counted with multiplicities) which are solutions of $z^2 - z \lambda_k -\eta=0$ for  $k \in [n]$.
\end{lemma}

Similar exact relations have also been used when the graph has a very specific structure, like bipartite biregular (see \cite{brito2018spectral}). When the graph $G$ does not exhibit such a simple structure, the relation between $A$ and $B$ becomes more involved. Several Ihara-Bass-like formulas are available (see for instance \cite{watanabe2009graph,benaych2017spectral,anantharaman2017some}), but they are usually hard to analyze. 

As cleverly noted in \cite{wang2017limiting}, the spectrum of $H$ as in \eqref{def:H} is hard to describe in terms of the spectrum of $A$, but the spectrum of $ \hat H_0$ in the preceding lemma is completely explicit in terms of the spectrum of $\hat{A}$, even if $\hat A$ has no specific structure. It is therefore quite natural to study the spectrum of $\hat H_0$ using the spectrum of $\hat A$, then use perturbation theorems to infer results on the `true' non-backtracking spectrum, the spectrum of $H$. This is done in \cite{wang2017limiting} through a combination of the Bauer-Fike theorem and a refinement of the Tao-Vu replacement principle \cite{tao2010random}.

The celebrated Bauer-Fike theorem says that\textit{ if a square matrix $\hat A$ is diagonalizable, say $\hat A=P\Delta P^{-1}$ for a diagonal matrix $\Delta$ and a non-singular matrix $P$, then under a perturbation $E$, every eigenvalue of the matrix $\hat A+E$  is within distance $\varepsilon$ of an eigenvalue of $\hat A$}, where $\varepsilon = \kappa(P)\Vert E \Vert$, and $\kappa(P)=\Vert P \Vert \Vert P^{-1} \Vert$ is the condition number (see for instance \cite{bordenave2018nonbacktracking}). 

 {We observe that} the Bauer-Fike theorem, while optimal in the worst case, is indeed extremely wasteful when applied to $H$ and $H_0$. Taking into account the specific structure of $H$ and $H_0$ yields a better perturbation bound at virtually no cost, as shown in the next section. 
\subsection{Bauer-Fike theorems for quadratic eigenvalue problems}

A \emph{quadratic eigenvalue problem (QEP)} consists of finding the zeroes of the polynomial equation 
\begin{equation}\label{QEP}0= \det (z^2 M - z \hat{A} - X)\end{equation}
 where $M,\hat 
 A,X$ are square matrices,  and $M$ is non-singular. Such problems appear in a variety of contexts and  there exists an extensive literature on them, mainly from a numerical point of view (see the survey \cite{tisseur2001quadratic}). 

The triplet $(M, \hat A,X)$ can be replaced with by the triplet $(I, \hat A M^{-1}, XM^{-1})$ without changing the problem, so we will be interested in the case where $M=I$. In this case, one can easily check that the solutions of \eqref{QEP} are the eigenvalues of the $2n \times 2n$ matrix
\[Q_{ \hat A,X}:= \begin{bmatrix}
\hat A & X \\I & 0
\end{bmatrix}, \]
which is called a \emph{linearization} of the problem. In this section, we will present extensions of the Bauer-Fike theorem for linearizations of quadratic eigenvalue problems. 

If both matrices $\hat A$ and $X$ are diagonal, say $\hat A = \mathrm{diag}(\hat a_i)$ and $X=\mathrm{diag}(x_i)$, then it is easily seen through elementary linear algebra operations that 
\begin{equation}\label{eq:QAXdet}
\det(Q_{\hat A,X}  -z I) = \prod_{i=1}^n (z^2 - \hat a_i z - x_i)
\end{equation}
and the eigenvalues of $Q_{\hat A,X}$ are the $2n$ complex solutions of the collection of $n$ quadratic equations $z^2-\hat a_iz - x_i=0$ for $1\leq i\leq n$. 

We say the matrices $\hat A$ and $X$ are \textit{co-diagonalizable} if there is a common non-singular matrix $P$ such that $P\hat{A}P^{-1}$ and $PXP^{-1}$ are diagonal. If $\hat{A}, X$ are co-diagonalizable, then the identity \eqref{eq:QAXdet} still holds with $\hat a_i$ being the eigenvalues of $ \hat A$ and $x_i$ being those of $X$. As a consequence, we say that that $Q_{ \hat A,X}$ is \emph{QEP-diagonalizable} if $\hat A$ and $X$ are co-diagonalizable. This is equivalent to ask that the matrix $z^2I - z\hat A-X$ is diagonalizable for any $z\in\mathbb C$. 

\bigskip

Our main tool for the perturbation analysis is the following theorem.

\begin{theorem}\label{thm:BFQEP}
Let $\hat{A},\hat{B}, X, Y$ be $n \times n$ matrices. We define
\begin{equation*}
L_0=\begin{bmatrix}
\hat{A} & X \\ I & 0
\end{bmatrix}, \qquad L = \begin{bmatrix}
\hat B & Y \\ I & 0
\end{bmatrix}.
\end{equation*}
Suppose $L_0$ is QEP-diagonalizable, with $\hat A$ and $X$ being diagonalized by the common matrix $P$. Then, for any eigenvalue $\mu$ of $L$, there is an eigenvalue $\nu$ of $L_0$ such that 
\begin{equation}\label{eq:QEPBF:1}
|\mu - \nu | \leqslant \sqrt{\kappa(P)} \sqrt{\Vert X-Y + \mu(\hat{A}-\hat B) \Vert}.
\end{equation}
Moreover, `multiplicities are preserved' in the following sense:
Denote $\varepsilon(\mu)$ the RHS of \eqref{eq:QEPBF:1} and $\varepsilon=\max_{\mu \in \spec(L)}\varepsilon(\mu)$. If $\nu_1, \dotsc, \nu_n$ are the eigenvalues of $L_0$ and  $\mathcal{K}$ is a subset of $[n]$ such that 
\begin{equation*}
\left( \cup_{j \in \mathcal{K}} \mathcal{B}(\nu_k, \varepsilon) \right) \cap \left( \cup_{j \notin \mathcal{K}}\mathcal B(\nu_k, \varepsilon) \right)= \varnothing,
\end{equation*}
where $\mathcal B(\nu_k,\varepsilon)=\{z\in \mathbb C: |z-\nu_k|\leq \varepsilon \}$ for $1\leq k\leq n$,
then the number of eigenvalues of $L$ in $\cup_{j \in \mathcal{K}}\mathcal{B}(\nu_k, \varepsilon) $ is exactly equal to $|\mathcal{K}|$.
\end{theorem}
\begin{remark} Theorem \ref{thm:BFQEP} is stated for two general matrices $L_0$ and $L$. However, the inequality \eqref{eq:QEPBF:1} will yield good perturbation bound only  when  we can control the difference between $\hat{A}, \hat B$, the difference between $X,Y$, and the condition number $\kappa (P)$.
\end{remark}

\begin{proof}
Assume $\mu$ is an eigenvalue of $L$. The matrix $R_\mu := \mu^2I - \mu \hat B -Y$ is then singular. Assume, in addition, that $\mu$ is not an eigenvalue of $L_0$. Then, the matrix $S_\mu := \mu^2I - \mu \hat{A} -X$ is non-singular. We have 
\begin{align*}
R_\mu = \mu^2I - \mu \hat B -Y &= \mu^2 I - \mu \hat{A} - X + X + \mu \hat{A}  - \mu \hat B -Y \\
&= S_\mu  + X + \mu \hat{A}  - \mu \hat B -Y \\
&= S_\mu (I +S_\mu^{-1} (X + \mu \hat{A}  - \mu \hat B -Y) ).
\end{align*}
As a consequence the matrix $I +S_\mu^{-1} (X + \mu \hat{A}  - \mu \hat B -Y) $ is singular, which directly implies that $-1$ is an eigenvalue of $S_\mu^{-1} (X + \mu \hat{A}  - \mu\hat B -Y)$. therefore by the definition of spectral norm,
$$1\leqslant \Vert S_\mu^{-1} \Vert \cdot  \Vert X + \mu \hat{A}  - \mu \hat B -Y \Vert = \Vert S_\mu^{-1} \Vert\cdot \Vert X-Y + \mu (\hat{A}-\hat B) \Vert.$$
As noted before the statement of the theorem, if $L_0$ is QEP-diagonalizable then the matrix $S_\mu$ is indeed diagonalizable: if $\Sigma=\diag(\lambda_i)$ is the diagonal matrix of eigenvalues of $\hat{A}$, $\Delta=\diag(\delta_i)$ the diagonal matrix of eigenvalues of $X$, and $P$ their common diagonalization matrix, then $S_\mu = P^{-1}(\mu^2 I - \mu \Sigma - \Delta)P$, and the eigenvalues of $S_\mu$ are the complex numbers $\mu^2 - \mu \lambda_k - \delta_k$, so  
\[\Vert S_\mu^{-1} \Vert \leqslant \kappa(P) \times \max_{k \in [n]} |\mu^2 - \mu \lambda_k - \delta_k |^{-1}. \]
From this, we infer that there is a $k \in [n]$ such that $\Vert S_\mu^{-1}\Vert \leqslant \kappa(P) \times |\mu^2 - \mu \lambda_k - \delta_k|^{-1}$.  Let us denote $\alpha_k$ and $\beta_k$ the two complex solutions of $0=z^2 - z\lambda_k - \delta_k$ (they are eigenvalues of $L_0$), then
\[
	1\leqslant \kappa(P)\times  \frac{  \Vert X-Y + \mu (\hat{A}-\hat B) \Vert}{|\mu^2 - \mu \lambda_k - \delta_k|} =\kappa(P)\times \frac{  \Vert X-Y + \mu (\hat{A}-\hat B) \Vert}{|\mu-\alpha_k||\mu-\beta_k|}
\]
which implies that
\[|\mu-\alpha_k||\mu-\beta_k|\leqslant \kappa(P)  \Vert X-Y + \mu (\hat{A}-\hat B) \Vert := x. \]
If $|\mu-\alpha_k|$ and $|\mu-\beta_k|$ were both strictly greater than $\sqrt{x}$, the preceding inequality would be violated. One of those distances is thus smaller than $\sqrt{x}$, thus proving \eqref{eq:QEPBF:1}.

The `multiplicities preserved' part is then proven as usual with the complex argument principle, see for instance \cite[Appendix A]{coste2017spectral}. 
\end{proof}

When applying the preceding result with $\hat A=\hat B$, one gets the following corollary. 

\begin{corollary}\label{cor:corol1}
Let 
\begin{equation}
L_0=\begin{bmatrix}
\hat{A} & X \\ I & 0
\end{bmatrix} \qquad L = \begin{bmatrix}
\hat{A} & Y \\ I & 0
\end{bmatrix}
\end{equation}
where $\hat{A},X,Y$ are square matrices and are such that $L_0$ is QEP-diagonalizable with $\hat{A}$ and $X$ diagonalized by the common matrix $P$. Then, for any eigenvalue $\mu$ of $L$, there is an eigenvalue $\nu$ of $L_0$ such that 
\begin{equation}\label{QEPBF:2}
|\mu - \nu | \leqslant \varepsilon := \sqrt{\kappa(P)\Vert X-Y \Vert}.
\end{equation}
\end{corollary}

\begin{remark}[Comparison with classical Bauer-Fike]
Casting the classical Bauer-Fike theorem in this setting would yield an error term of $\varepsilon' = \kappa(Q)\Vert X-Y \Vert $, where $Q$ is the diagonalization matrix of $L_0$. We thus gain the whole square root, and we do not need to compute the condition number of $Q$. This improvement is remarkable when the matrix $\hat{A}$ is itself Hermitian, for in this case $P$ is unitary and $\kappa(P)=1$, thus reducing the error term to $\sqrt{\Vert X-Y \Vert}$. If we had invoked the classical Bauer-Fike theorem instead, the error term would be $\kappa(Q)\Vert X-Y \Vert$, which can be far bigger than $\Vert X-Y\Vert$. In fact, the matrix $L_0$ is not Hermitian in general, and its diagonalization matrix $Q$ might be either difficult to compute or ill-conditioned: in \cite{wang2017limiting}, the bound obtained by the authors is $\kappa(Q) \leqslant O(\sqrt{1/p})$,  { where $p$ is the connection probability for an \ER graph $G(n,p)$}. Our version of the Bauer-Fike theorem shows that for QEP, the only parameters at stake in perturbations are those of the original matrices $\hat{A}$ and $X$, not those of the linearization of the QEP. 
\end{remark}

\section{The stochastic block model in the logarithmic regime}\label{sec:SBM}

 {In this section, we collect results from the literature on stochastic block models or inhomogeneous Erd\H{o}-R\'enyi graphs, based on which we prove several quick results for our models, as given in Proposition \ref{thm:concentrateA}, Proposition \ref{thm:specH_0} and Corollary \ref{cor:corol_inv}. }

\subsection{Outliers of the adjacency matrix}

The concentration of the spectral norm for the SBMs follows immediately from the  spectral norm bounds given in \cite{latala2018dimension,benaych2017spectral} for inhomogeneous
random matrices and random graphs. Recall Assumption \eqref{condition:sparsity}.
The following statement can be found for example in Example 4.1 of \cite{latala2018dimension}: \emph{assume $\alpha=\omega(\log n)$, then 
	\begin{align*}
	\mathbb E[ \|A-\mathbb EA\|] \leqslant (2+o(1))\sqrt{\alpha}.
	\end{align*}}
Also from Equation (2.4) in \cite{benaych2017spectral}, there exists a constant $c>0$ such that
\begin{align*}
\mathbb P\left( \left|\frac{\|A-\mathbb EA\|}{\sqrt{\alpha}}- \frac{\mathbb E[\|A-\mathbb EA\|]}{\sqrt{\alpha}}\right|\geqslant t\right)	\leqslant 2e^{-c\alpha^2t^2}.
\end{align*}
Taking  $t=\sqrt{\log n}/\alpha$ in the inequality above,  we have with probability $1-2n^{-c}$ that 
\begin{align}
\|A-\mathbb EA\| &\leqslant \sqrt{\frac{\log n}{\alpha}} +\mathbb E[\|A-\mathbb EA\|] \leqslant (2+o(1))\sqrt{\alpha}.
\end{align}
Since all the eigenvalues of $\mathbb E [A]$ are $\{-p,\beta,\alpha\}$, and $p=o(1)$ under assumption \eqref{condition:sparsity}, the Weyl eigenvalue inequalities for Hermitian matrices yields the following proposition.

\begin{proposition}\label{thm:concentrateA}
Assume $\alpha=\omega(\log n)$, then with high probability the following holds:
\begin{align*}
|\lambda_1(A)-\alpha|  &\leqslant (2+o(1))\sqrt{\alpha}, \quad 
|\lambda_2(A)-\beta|\leqslant (2+o(1))\sqrt{\alpha},\\
\max_{k\geqslant 3}|\lambda_k(A)| &\leqslant  (2+o(1))\sqrt{\alpha}.
\end{align*} 
\end{proposition}

\subsection{Spectrum  of the partially derandomized matrix $H_0$}

We will use the notation $$\gamma =  \frac{n(p+q)}{2}-p-1 = \alpha-1$$ which is the `mean degree minus one'. We introduce the \textit{partial derandomization} of $H$ (defined in \eqref{def:H}) as:
\begin{align}	\label{eq:defH_0}
H_0=\left[\begin{matrix}
	A &-\gamma I\\
	I & 0
\end{matrix}\right].
\end{align}

As already mentioned in Lemma \ref{lem:H0}, by elementary operations on  $H_0$, one finds that the characteristic polynomial of $H_0$ is indeed equal to
\begin{align}
\chi_0(z)&=\prod_{i=1}^n (z^2 - \lambda_i z + \gamma ). \label{eq:detHzero}
\end{align}
The eigenvalues of $H_0$ hence come into conjugate pairs coming from eigenvalues of $A$. Those eigenvalues $\lambda_k$ of $A$ for which $|\lambda_k|<2\sqrt{\gamma}$ give rise to two complex conjugate eigenvalues 
\begin{equation}\label{eigenvaluesHzero:complex}
 \frac{\lambda_k - \ic \sqrt{ 4\gamma - \lambda_k^2}}{2} \quad \text{ and } \quad  \frac{\lambda_k + \ic \sqrt{ 4\gamma - \lambda_k^2}}{2}
\end{equation}
and the other ones, the outliers $|\lambda_k|>2\sqrt{\gamma}$ of the spectrum of $A$, give rise to  two `harmonic conjugate' eigenvalues 
\begin{equation}\label{eigenvaleus Hzero:real}
 \frac{\lambda_k + \sqrt{\lambda_k^2- 4\gamma }}{2} \qquad \text{and} \qquad  \frac{\gamma}{ \frac{\lambda_k + \sqrt{\lambda_k^2- 4\gamma }}{2} }.
\end{equation}

Next we obtain a description of the eigenvalues of $H_0$ from the discussion on the spectrum of $A$ in Proposition \ref{thm:concentrateA}. The description is illustrated in the second panel of Figure \ref{figure2}, the first one depicting the same phenomenon but for \ER graphs (with only one outlier in the spectrum).
\begin{figure}\centering
\includegraphics[width=\textwidth]{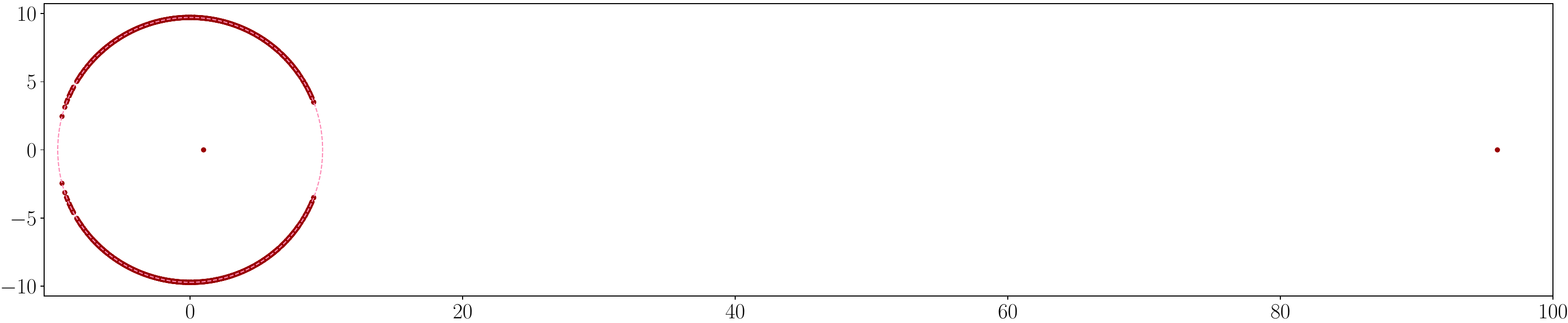}
\includegraphics[width=\textwidth]{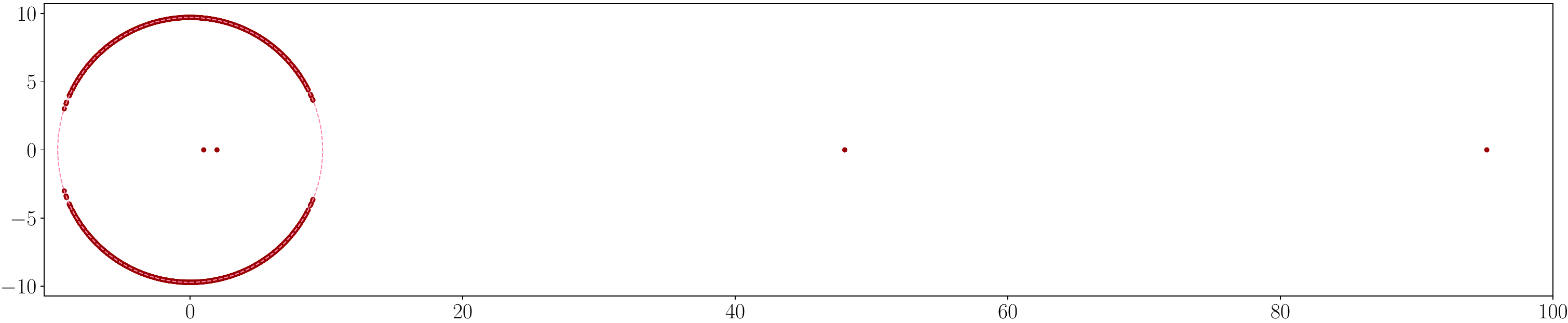}
\caption{The spectrum of $H_0$ for the two graphs whose non-backtracking spectrum is depicted in Figure \ref{figure}. The circle of radius $\sqrt{\alpha-1}$ is in pink. }\label{figure2}
\end{figure}

\begin{proposition}\label{thm:specH_0}
Under assumption \eqref{condition:sparsity}  with high probability the following holds for $H_0$.
	\begin{enumerate}
		\item The two eigenvalues with greater modulus, $\lambda_1(H_0)$ and $\lambda_2(H_0)$, are real, and they satisfy 
		\[\lambda_1(H_0) =  \alpha+O(
		\alpha^{3/4}) \qquad \text{and} \qquad \lambda_2(H_0) = \beta+O(\alpha^{3/4}). \]
		\item The two eigenvalues with smaller modulus, $\lambda_{2n}(H_0)$ and $\lambda_{2n-1}(H_0)$, are real, and they satisfy
\[\lambda_{2n-1}(H_0)=\frac{\alpha}{\beta}+O(\alpha^{-1/4}),\quad \lambda_{2n}(H_0)=1+O(\alpha^{-1/4}).\]
\item All the other $2n-4$ eigenvalues have modulus smaller than $ \sqrt{\alpha} +o(\sqrt \alpha )$. Among them, complex eigenvalues lie on a circle of radius $\sqrt{\alpha-1}$ and real ones lie in the intervals $[\sqrt\alpha-o(\sqrt{\alpha}),\sqrt\alpha+o(\sqrt{\alpha}) ]$ and $[-\sqrt\alpha-o(\sqrt{\alpha}),-\sqrt\alpha+o(\sqrt{\alpha})]$.
	\end{enumerate}
\end{proposition}

\begin{proof}
We use Proposition 
\ref{thm:concentrateA} and the link described before between the spectrum of $A$ and the spectrum of $H_0$  {given in \eqref{eigenvaluesHzero:complex} and \eqref{eigenvaleus Hzero:real}.}

The greatest eigenvalue of $A$ is
$\lambda_1=\lambda_1(A)=\alpha+O(\sqrt{\alpha})$,   from \eqref{eigenvaluesHzero:complex}, it gives rise to two real eigenvalues of $H_0$:
\begin{align*}\mu_1&=\frac{\lambda_1+\sqrt{\lambda_1^2-4\gamma}}{2}= \alpha+O(\alpha^{3/4}), \qquad  \mu_2 =\frac{\alpha-1}{\mu_1}=1+O(\alpha^{-1/4}).
\end{align*}

The second greatest eigenvalue ( {in absolute value}) of $A$ is $\lambda_2=\lambda_2(A)=\beta+O(\sqrt{\alpha})$,  which gives rise to two real eigenvalues of $H_0$:
\begin{align*}
	\mu_3 &=\frac{\lambda_2+\sqrt{\lambda_2^2-4\gamma}}{2}=\beta+O(\alpha^{3/4}),\quad \mu_4 =\frac{\alpha-1}{\mu_3}=\frac{\alpha}{\beta}+O( \alpha^{-1/4}).
\end{align*}

For the eigenvalues $\lambda_k$ of $A$ with $2\sqrt{\gamma}<|\lambda_k|\leqslant (2+o(1))\sqrt{\alpha}$, from \eqref{eigenvaleus Hzero:real}, the same argument gives
\begin{align}\label{eq:mu2k}
	|\mu_{2k}|&=\left|\frac{\lambda_k+\sqrt{\lambda_k^2-4\gamma}}{2}\right| = (1+o(1))\sqrt{\alpha},\qquad 
|\mu_{2k+1}|=\left|\frac{\alpha-1}{\mu_{2k}}\right|=(1+o(1))\sqrt{\alpha}.	
\end{align}

Finally, from \eqref{eigenvaluesHzero:complex}, all eigenvalues of $A$ with $|\lambda_k|\leqslant 2\sqrt{\gamma}$ give rise to two complex conjugate eigenvalues of $H_0$ with magnitude $\sqrt{\gamma}=\sqrt{\alpha-1}$. This completes the proof.
\end{proof}

The preceding description in Proposition \ref{thm:specH_0} also shows that $H_0$ is non-singular, and we can  quickly describe the eigenvalues of $H_0^{-1}$.
 {  We will later show that the norm of $(H^{-1}-H_0^{-1})$ is very small in Section \ref{sec:proof}. The strategy is then to apply  Theorem \ref{thm:BFQEP} to $H_0^{-1}$ and $H^{-1}$, which gives a more  precise estimate on the location of the outliers in $H$. See Remark \ref{remark:comparision} for further discussion.}

\begin{corollary}[inverse spectrum] \label{cor:corol_inv}
Under the  assumption \eqref{condition:sparsity}, with high probability, in the spectrum of the matrix $H_0^{-1}$ there are exactly two real outliers 
\begin{align}
 \zeta_1 &= \frac{1}{1+O(\alpha^{-1/2})}=1+o(1), \label{eq:zeta1}\\
\zeta_2 &= \frac{1}{\frac{\alpha}{\beta}+O( \alpha^{-1/2})}=\frac{\beta}{\alpha}+o(1),	 \label{eq:zeta2}
\end{align}
and all the other eigenvalues of $H_0^{-1}$ have modulus smaller than $(\sqrt{\alpha}(1+o(1))^{-1}=o(1)$. 
\end{corollary}
\begin{proof}
 {The location of the two real outliers in the spectrum of $H_0^{-1}$ comes from part (2) in Proposition \ref{thm:specH_0}. The location of all the other eigenvalues comes from part (1) and (3) in Proposition \ref{thm:specH_0}.}
\end{proof}
We now turn to the description of the global behavior of the spectrum of $A$.

\subsection{Limiting spectral distribution of A}

If we have an SBM with two  blocks of equal size, and $p,q=\omega(1/n)$, the empirical spectral distribution of $\frac{A}{\sqrt{\alpha}}$ will converge weakly to the semicircle law: for any bounded continuous test function $f$,  almost surely
\begin{equation}\label{convergence}
\frac{1}{n}\sum_{i=1}^n f(\lambda_i(A/\sqrt{\alpha})) \to \frac{1}{2\pi}\int_{-2}^2 f(t)\sqrt{4-t^2}\mathrm{d}t. 
\end{equation}
 This can be seen from the graphon representation of SBMs and the result for generalized Wigner matrices (Section 4 in \cite{zhu2018graphon}), since each row in $\mathbb EA$ has the same row sum or equivalently, each vertex has the same expected degree. If the degree is not homogeneous, then the limiting spectral distribution will not be the semicircle law.  We recall the following result from \cite{zhu2018graphon} for generalized Wigner matrices, which includes the  regime where the sparsity parameter is $\omega(1/n)$.
 
\begin{theorem}[Theorem 4.2. in \cite{zhu2018graphon}] Let $A_n$ be a random Hermitian matrix such that entries on and above the diagonal are independent and satisfy the following conditions:
\begin{enumerate}
	\item $\mathbb E[a_{ij}]=0,\quad \mathbb E|a_{ij}|^2=s_{ij}$.
	\item $\frac{1}{n}\sum_{j=1}^n s_{ij}=1+o(1)$ for all $i\in [n]$.
	\item For any constant $\eta>0$, $\displaystyle \lim_{n\to\infty}\frac{1}{n^2}\sum_{1\leqslant i,j\leqslant n} \mathbb E[|a_{ij}|^2 \mathbf{1}(|a_{ij}|\geqslant \eta\sqrt{n})]=0.$
	\item $\sup_{ij}s_{ij}\leqslant C$ for a constant $C>0$. 
\end{enumerate}
Then the empirical spectral distribution of $\frac{A_n}{\sqrt n}$ converges weakly to the semicircle law almost surely, which means that on an event with probability $1$, the convergence \eqref{convergence} holds for any bounded continuous function $f$.
\end{theorem}
We obtain the following theorem for the adjacency matrix $A$ of an SBM, and also for $H_0$. 

\begin{theorem}\label{thm:bulkH0}
Assume $\alpha\to\infty$ and $p,q\to 0$.  The empirical spectral distribution of $\frac{A}{\sqrt{\alpha}}$ converges weakly to the semicircle law supported on $[-2,2]$ almost surely.	

Moreover, the empirical spectral distribution of $\frac{H_0}{\sqrt \alpha}$ converges weakly  almost surely to a distribution on the circle of radius $1$, and the limiting distribution of the real part of the eigenvalues of $H_0$ is the semicircle law rescaled on $[-1,1]$.
\end{theorem}

\begin{proof}
We first consider the centered and scaled  matrix \[M := (m_{ij})_{1\leqslant i,j\leqslant n}=\frac{A-\mathbb EA}{\sqrt {(p+q)/2}}.\]
For $i\not=j$  we have
\[ \mathbb E[|m_{ij}|^2]=\begin{cases}
	\frac{q(1-q)}{(p+q)/2} & \text{if $i,j$ are in the same block,}\\
	\frac{p(1-p)}{(p+q)/2} &\text{otherwise.}
\end{cases} \\
\]
Then for all $i\in [n]$, \begin{align*}
	\frac{1}{n}\sum_{j}\mathbb E[|m_{ij}|^2]&=\frac{1}{n}\left( \left(\frac{n}{2}-1\right)\frac{p(1-p)}{(p+q)/2}+\frac{n}{2}\frac{q(1-q)}{(p+q)/2}\right) =1+o(1).
\end{align*}

 One can quickly check that all the conditions in Theorem 3.4 hold for $M$. Therefore the empirical spectral distribution of $$\frac{M}{\sqrt n}=\frac{A-(\mathbb EA+p I)}{\sqrt{n(p+q)/2}}+\frac{pI}{\sqrt{n(p+q)/2}}$$
 converges weakly to the semicircle law. Or equivalently the empirical spectral distribution of $ \frac{A-(\mathbb EA+pI)}{\sqrt{\alpha}}$ converges weakly almost surely to the same distribution. Finally, since the rank of the matrix $(\mathbb EA+pI)$ is $2$ and  by the Cauchy interlacing theorem  for the eigenvalue of Hermitian matrices, the empirical spectral distribution of $\frac{A}{\sqrt {\alpha}}$ converges weakly to the semicircle law almost surely.

\bigskip

We now turn to the second part of the theorem. Note that from \eqref{eq:detHzero}, one eigenvalue $\lambda_i(A)$ corresponds to two eigenvalues of $H_0$ momentarily denoted by $\mu_{2i-1}(H_0),\mu_{2i}(H_0)$, and such that 
\[ 
	\RE~ \mu_{2i-1}(H_0)=\RE~ \mu_{2i}(H_0)=\frac{\lambda_i(A)}{2}.
\]
The empirical spectral distribution of the real parts of eigenvalues of $H_0/\sqrt{\alpha}$ satisfies
$$
	\frac{1}{2n}\sum_{j=1}^{2n}\delta_{\RE{}(\mu_j/\sqrt{\alpha})} =\frac{1}{n}\sum_{i=1}^n \delta_{\lambda_i/2\sqrt{\alpha}},
$$
which converges weakly almost surely to the semicircle law rescaled on $[-1,1]$ by the first part of the theorem.
\end{proof}

\subsection{Concentration of the degrees}\label{sec:degrees}

We finally describe the degrees in the SBM. Let us note $d_i$ the degree of vertex $i$ in the SBM graph.  {Under the assumption \eqref{condition:sparsity}, we have $\alpha=\omega(\log n)$}, and the degrees are highly concentrated in the following sense.
\begin{lemma}\label{lemma:normE}
With high probability $\displaystyle \max_{i \in [n]}|d_i-\alpha|=o(\alpha).$	
\end{lemma}
\begin{remark}
 Note that Lemma \ref{lemma:normE} is no longer true in other regimes.  {When $\alpha=O(\log n)$, the event $d_i=(1+o(1))\alpha$ for all $i\in [n]$ does not happen with high probability (see for example \cite[Chapter 3]{bollobas2001random}). Then the diagonal degree matrix $D$ is not close to $\alpha I$, which is a barrier for our perturbation analysis to work.}    
\end{remark}

Lemma \ref{lemma:normE} can be found in the literature, but we provide a proof for completeness. We  recall Bernstein's inequality: \emph{let $Y_n=\sum_{i=1}^n X_i$ where $X_i$ are independent random variables such that $|X_i|\leqslant b$. Define $\sigma_n^2 := \mathrm{Var}(Y_n)$. Then for any $x>0$,
	\begin{align*}
	\mathbb P(|Y_n-\mathbb E[Y_n]|\geqslant x)\leqslant 2\exp \left(\frac{-x^2}{2(\sigma_n^2+b x/3)}\right).	
	\end{align*}
	}
Now we prove Lemma \ref{lemma:normE}.
\begin{proof} Each $d_i$ has the same distribution with mean $\alpha$, hence we can apply the union bound and get
\begin{align}\label{eq:normE_proof}
\mathbb P\left(\max_{i\in [n]}|d_i-\alpha|\geqslant x \right)\leqslant \sum_{i \in [n]}\mathbb{P}(|d_i-\alpha|\geqslant x) = n \mathbb{P}(|d_1 - \alpha|\geqslant x).
\end{align}
Let us write $d_1 = X_2+\dotsb + X_{n/2}+X_{n/2+1}+\dotsb + X_n$, where the $X_i$ are independent, and $X_i$ is a Bernoulli random variable with parameter $p$ if $i \in \{2, \dotsc, n/2 \}$ and $q$ if $i \in \{n/2+1, \dotsc, n\}$. Those variables are all bounded by $1$ so we can take $b=1$ in Bernstein's inequality. The variance is
\[
	\sigma_n^2= \left(\frac{n}{2}-1 \right)p(1-p)+\frac{n}{2}q(1-q)\leqslant \alpha.
	\] 
	From  \eqref{eq:normE_proof} and Bernstein's inequality we have
	\begin{align*}
	 \mathbb P( |d_1-\alpha|\geqslant x)\leqslant 2\exp \left( \frac{-x^2}{2(\alpha+x/3)} \right).
	\end{align*}
Let $h(n)$ be any sequence of positive numbers. The choice $x=\frac{\alpha}{h(n)}$ then leads to 
	\begin{align*}
	\mathbb P\left(\max_{i\in[n]}|d_i-\alpha|\geqslant x\right) \leqslant 2\exp \left(\log n - \frac{\alpha}{2h^2(n)+2h(n)/3}\right).		
	\end{align*}
	
Since we know $\alpha=\omega(\log n)$, any choice of $h(n)$ growing to $\infty$ slowly enough will be sufficient; for instance if $\alpha=\log(n)f(n)$ with $f(n) \to \infty$, we take $h(n)=f(n)^{1/3}$ and we obtain that $\max_{i\in[n]}|d_i-\alpha| = o(\alpha)$ with probability $1-o(1)$.
\end{proof}

\section{Proof of Theorem \ref{thm:main}}\label{sec:proof}

\subsection{Existence of bulk insiders}
In this section we prove the existence of the isolated eigenvalues inside the bulk. To do this, we compare the spectrum of $H$ (defined in \eqref{def:H}) and $H_0$ (defined in \eqref{eq:defH_0}). We also need to compare the spectrum of $H^{-1}$ and $H_0^{-1}$ to have a more refined estimate compared to \cite{wang2017limiting}. See Remark \ref{remark:comparision} for further discussion.

Fix any non-singular square matrix $X$. One can easily check that 
\begin{align}\label{eq:invert}
\begin{bmatrix}
A & -X \\ I & 0
\end{bmatrix}^{-1}=\begin{bmatrix}
0  & I  \\ -X^{-1} & X^{-1} A
\end{bmatrix}
\end{align}
and by conjugation  the spectrum of this matrix is the same as the spectrum of the matrix 
\[
	\begin{bmatrix}
	X^{-1}A & - X^{-1} \\ I & 0
	\end{bmatrix}.
\]
Let us introduce the matrices
\begin{equation}\label{def:K}
K = \begin{bmatrix}
(D-I)^{-1} A & -(D-I)^{-1} \\ I & 0 
\end{bmatrix}\qquad \text{and} \qquad  K_0 = \begin{bmatrix}
(\alpha - 1)^{-1} A & -(\alpha-1)^{-1}I \\ I & 0 
\end{bmatrix}.
\end{equation}
These matrices have the same spectrum as (respectively) $H^{-1}$ and $H_0^{-1}$ and we are going to apply Theorem \ref{thm:BFQEP} to them. First, one has to note that the spectrum of $K$ is indeed bounded away from zero. More precisely, all the eigenvalues of $H$ are bounded below by $1$, as explained in the following  statement (Theorem 3.7 in \cite{angel2015non}, the same result for finite graphs was first given in \cite{kotani2000zeta}): \emph{let $d_{\min}\geq  2$ and $d_{\max}$ be the minimal and maximal degrees of some finite or infinite graph $G$. Then the spectrum of $B$ is included in
	\begin{align}
 \left\{ \lambda\in \mathbb C\setminus \mathbb R: \sqrt{d_{\min}-1}\leqslant |\lambda| \leqslant \sqrt{d_{\max}-1} \right\}\cup \{\lambda\in \mathbb R: 1\leqslant |\lambda|\leqslant d_{\max}-1	\}.
	\end{align}
}

We see from Lemma \ref{lemma:normE} that with high probability all the degrees in our graph are greater than $2$, hence every eigenvalue of $H$ has modulus greater than $1$, thus ensuring that every eigenvalue $\mu$ of $H^{-1}$ has $|\mu|\leqslant 1$. We now apply Theorem \ref{thm:BFQEP} to $K$ and $K_0$. It is easily seen from \eqref{def:K} that $K_0$ is QEP-diagonalizable and the change-of-basis matrix $P$ is unitary since $A$ is Hermitian. We take
 \begin{align}\label{eq:XYdef}
 X=(\alpha-1)^{-1}I,\quad \text{ and }\quad  Y=(D-I)^{-1}.
 \end{align} From Theorem \ref{thm:BFQEP}, we have
\begin{align*}
\varepsilon := \max_{\mu \in \spec(H^{-1})}\varepsilon (\mu) &= \max_{\mu \in \spec(H^{-1})}\sqrt{\kappa(P)} \sqrt{\Vert X-Y + \mu(Y-X)A \Vert}\\
&\leqslant \max_{\mu \in \spec(H^{-1})} \sqrt{\Vert X-Y \Vert (1+ |\mu| \Vert A \Vert)} \\
&\leqslant \max_{\mu \in \spec(H^{-1})} \sqrt{\Vert X-Y \Vert (1+  \Vert A \Vert)} \\
&\leqslant \sqrt{\Vert X-Y\Vert(1+\alpha(1+o(1))}, 
\end{align*}
where the last line holds with high probability from the description of the spectrum of $A$ in Proposition \ref{thm:concentrateA}. It turns out that $\varepsilon=o(1)$, as a consequence of the following lemma.

\begin{lemma} For $X$ and $Y$ defined in \eqref{eq:XYdef}, with high probability $\Vert X -Y \Vert = o(\alpha^{-1})$. 
\end{lemma}

\begin{proof}Since $X,Y$ are diagonal matrices, we have 
\[
	\Vert X-Y \Vert= \max_{i \in [n]} \left|\frac{1}{d_i-1} - \frac{1}{\alpha-1}\right|=	\max_{i \in [n]} \frac{|d_i-\alpha|}{(d_i-1)(\alpha-1)} .
\]
By Lemma \ref{lemma:normE}, with high probability $\max_{i \in [n]}|d_i-\alpha|=o(\alpha)$ and this implies the lemma.
\end{proof}

We thus have $\varepsilon=o(1)$ and  now we can combine the `multiplicities preserved' part of Theorem \ref{thm:BFQEP} and the description of the spectrum of $K_0$ in Corollary \ref{cor:corol_inv}.

\begin{remark}\label{rmk:sparsity}
 {Recall $\zeta_1,\zeta_2$ from Corollary \ref{cor:corol_inv}.
  The crucial fact here is that $\zeta_1\approx 1$ and $\zeta_2 \approx \beta/\alpha$ are of order $1$ and in particular they are bounded away from $0$, which is guaranteed by the third inequality in our assumption \eqref{condition:sparsity}.}
\end{remark}

Theorem \ref{thm:BFQEP} implies that there is exactly one eigenvalue of $K$ in $\mathcal B(\zeta_1, \varepsilon)$, one in $\mathcal B(\zeta_2, \varepsilon)$ and all the other eigenvalues have modulus $o(1)$. In other words, there are exactly two eigenvalues $\xi_1, \xi_2$ of $H$ such that 
\begin{align*}
   \xi_1^{-1} = 1+o(1),\quad \xi_2^{-1} = \frac{\beta}{\alpha}+o(1),  
\end{align*}
and all the other ones have inverse modulus $o(1)$.  By the continuity of $x \mapsto x^{-1}$, we have exactly two eigenvalues of $H$,
\begin{align}\label{eq:omega1}
    \xi_1=1+o(1), \quad \xi_2=\frac{\alpha}{\beta}+o(1),
\end{align}
which are of order $1$,
and all the other eigenvalues of $H$  have inverse modulus $\omega(1)$.

Since $0$ is always an eigenvalue of the Laplacian $D-A$, we have $$1\in \{z\in \mathbb C: \det(z^2I-zA+D-I)=0\},$$ 
which implies $1$ is always an eigenvalue of $H$.
So $\xi_1$ is indeed exactly equal to $1$,  {otherwise we have three eigenvalues of $H$: $1, \xi_1$ and $\xi_2$ that are of order $1$,  a contradiction to \eqref{eq:omega1}}.   

Moreover, $\xi_2$ must be a real eigenvalue of $H$, otherwise from the fact that the spectrum of $B$ is symmetric with respect to the real line, we would see two eigenvalues of $K$ in the  ball $\mathcal B(\zeta_2,\varepsilon)$, which is a contradiction to Theorem \ref{thm:BFQEP}. This completes the proof of \eqref{eq:insiders}.

\subsection{Existence of the outliers}We now simply apply Theorem \ref{thm:BFQEP} to the matrices $H$ defined in \eqref{def:H} and $H_0$ defined in \eqref{eq:defH_0}. Here, $A=B$ and in fact we are in the setting of Corollary \ref{cor:corol1} with $X=(\alpha-1)I$ and  $Y=D-I$. Hence 
\[\varepsilon = \sqrt{\Vert X-Y\Vert}=\sqrt{\max_{i \in [n]}|d_i-\alpha|}=o(\sqrt{\alpha}) \]
with high probability from Lemma \ref{lemma:normE}. From the description of the spectrum of $H_0$ in Proposition \ref{thm:specH_0}, we see that there are two outliers located near $\beta$ and $\alpha$ and all other eigenvalues have order $O(\sqrt{\alpha})$. From this and  the `multiplicities preserved' part in Corollary \ref{cor:corol1}, we see that $H$ has two outliers located within distance $o(\sqrt{\alpha})$ of \[\lambda_1(H_0)=\alpha+O(\alpha^{3/4}), \quad \text{ and } \quad \lambda_2(H_0)=\beta+O(\alpha^{3/4}).\] By the symmetry of the spectrum with respect to the real line, those two outliers are real numbers. This completes the proof of \eqref{eq:outsiders}.

\begin{remark}\label{remark:comparision}
	Note that we could also use this strategy to infer the existence of the bulk insiders: in fact, the result would  yield the existence of two eigenvalues located in the balls $\mathcal B(1, \varepsilon)$ and $\mathcal B(\alpha/\beta, \varepsilon)$. These eigenvalues would be detached from the bulk of eigenvalues of $H$, which lie within distance $o(\sqrt{\alpha})$ of the circle of radius $\sqrt{\alpha}$; however, no further information can be inferred, since $o(\sqrt{\alpha})$ can  go to infinity as well. This is the reason why we had to compare $H^{-1}$ with $H_0^{-1}$, which has two effects: first, it isolates the two  `insiders' of $H_0$ and  the other eigenvalues close to zero; and secondly, it turns out that the norm of $H^{-1}-H_0^{-1}$ is very small.
In addition, our use of the specific Bauer-Fike theorem designed for QEP (Theorem \ref{thm:BFQEP}) yields more precise results than \cite{wang2017limiting}. 
\end{remark}

\subsection{Global spectral distribution}

We now prove the `bulk' part of Theorem \ref{thm:main}. The strategy is  the same as \cite{wang2017limiting} and we borrow their main theorem.

\begin{theorem}[Corollary 3.3. in \cite{wang2017limiting}]\label{thm:TaoVu}

	Let $M_m$ and $P_m$ be $m\times m$ matrices with entries in complex numbers, and let $f(z,m)\geq 1$ be a real function depending on $z,m$. Let $\mu_M$ be the empirical spectral distribution of any square matrix $M$. Assume that
	\begin{align}\label{eq:Cond1}
	\frac{1}{m}\|M_m\|_F^2+\frac{1}{m}\|M_m+P_m\|_F^2 	
	\end{align} is bounded in probability,
and 
\begin{align}\label{eq:Cond2}
	f(z,m) \|P_m\|\to 0
\end{align}
in probability, and  for almost every complex number $z\in\mathbb C$,
\begin{align}\label{eq:Cond3}
	\|(M_m-zI)^{-1}\|\leqslant f(z,m),
\end{align}
with probability tending to $1$, then $\mu_{M_m}-\mu_{M_m+P_m}$  converges in probability to zero.
\end{theorem}

Recall $H$ from \eqref{def:H} and $H_0$ from \eqref{eq:defH_0}. Take \[M_m=\frac{H_0}{\sqrt\alpha},\quad P_m=\frac{1}{\sqrt{\alpha}}(H-H_0)=\frac{1}{\sqrt{\alpha}}\begin{bmatrix}
 0 & \alpha I-D\\
 0 & 0	
\end{bmatrix}
\]
in Theorem \ref{thm:TaoVu}. 
If all the conditions in Theorem \ref{thm:TaoVu} hold, then the `bulk' part of Theorem \ref{thm:main} follows from our Theorem \ref{thm:bulkH0}.  

The condition \eqref{eq:Cond1} follows verbatim from the proof of Lemma 3.7. in \cite{wang2017limiting}. Condition \eqref{eq:Cond2} follows from Lemma 3.9. in \cite{wang2017limiting} and our Lemma \ref{lemma:normE}. Condition \eqref{eq:Cond3} follows from Lemma 3.9. in \cite{wang2017limiting}. This completes the proof of global spectral distribution part of Theorem \ref{thm:main}.

\subsection*{Acknowledgements} We would like to thank Lorenzo Dall'Amico for sharing the conjecture in \cite{dall2019optimized} and helpful discussions. We also thank Ke Wang for a careful reading of this paper and her useful suggestions.  We are  grateful to the organizers of the conference Random Matrices and Random Graphs at CIRM, during which this work was initiated. Y.Z. is partially supported by NSF DMS-1949617.

 \bibliographystyle{plain}
\bibliography{ref.bib}

\end{document}